\documentclass{amsart}

\usepackage{amssymb, amsmath, amsthm, pictexwd}

\newcommand {\C} [1] {{\mathbb C}^{#1}}

\newcommand {\R} [1] {{\mathbb R}^{#1}}

\newcommand {\pdd} [1] {\displaystyle{\frac{\partial }{\partial
{#1}}}}

\newtheorem{theorem}{Theorem}[section]

\newtheorem{lemma}[theorem]{Lemma}

\newtheorem{cor}[theorem]{Corollary}

\newtheorem{prop}[theorem]{Proposition}

\newtheorem{remark}[theorem]{Remark}

\numberwithin{equation}{section}

\begin{document}
\title{The Szeg\"o Kernel for Certain Non-Pseudoconvex Domains in $\C{2}$}
\author{Michael Gilliam and Jennifer Halfpap}\thanks{The second author was supported in part by NSF grant DMS-0654195.}
\maketitle

\section{Introduction}

Let $\Omega \subset \C{n}$ be a domain with smooth boundary
$\partial \Omega$.  Let $\mathcal{O}(\Omega)$ denote the space of
holomorphic functions on $\Omega$.  Associated with such a domain
are certain operators: the Bergman projection $\mathcal{B}$ and
the Szeg\"o projection $\mathcal{S}$. The former is the orthogonal
projection of $L^2(\Omega)$ onto the closed subspace $L^2(\Omega)
\cap \mathcal{O}(\Omega)$, while the latter is the orthogonal
projection of $L^2(\partial \Omega)$ onto the closed subspace
$\mathcal{H}^2(\Omega)$ of boundary values of elements of
$\mathcal{O}(\Omega)$. Ultimately, one would like results
concerning the mapping properties of these operators (e.g.,
conditions under which they extend to bounded operators on the
appropriate $L^p$ spaces).

Often, an understanding of these operators begins with an
investigation of an associated integral kernel. That is, often one
can identify distributions $B$ and $S$ so that for $f \in L^2
(\Omega)$ and $g\in L^2(\partial \Omega)$,
$$\mathcal{B}[f](z)=\int_{\Omega} f(w) B(z,w) \,dw$$
$$\mathcal{S}[g](z)=\int_{\partial \Omega}g(w) S(z,w) \,d\sigma(w).$$

Much is known about these kernels for pseudoconvex domains of
finite type. Kohn's formula connects the Bergman projection with
the $\bar{\partial}$-Neumann operator. Kerzmann \cite{Kerzman:72}
uses this connection to show that $B$ is equal to a $C^\infty$
function on $(\overline{\Omega}\times \overline{\Omega})\setminus
\Delta$, where $\Delta=\{\,(z,w)\in \partial \Omega \times
\partial \Omega : z=w\,\}$ is the diagonal of the boundary.

Much more is known about the operators in some settings; a number
of authors obtain sharp estimates on the kernels and their
derivatives near the diagonal as well as results on the mapping
properties of the operators. This is done, for example, by Nagel,
Rosay, Stein, and Wainger \cite{NRSW:89} for finite-type domains
in $\C{2}$ and by McNeal and Stein (\cite{McNeal:94},
\cite{McNealStein:94},\cite{McNealStein:97}) for convex domains in
$\C{n}$.

In contrast with the situation for pseudoconvex domains,
comparatively little is known about the Szeg\"o kernel for
non-pseudoconvex domains. A notable exception is the work of
Carracino (\cite{CarracinoPHD}, \cite{Carracino:07}), in which she
obtains detailed estimates for the Szeg\"o kernel on the boundary
of the (non-smooth) non-pseudoconvex domain of the form
\begin{equation}\label{tube domain}
\Omega =\{\,(z_1=x+iy,z_2=t+i\xi): \xi > b(x)  \,\}
\end{equation}
with
\begin{equation}\label{Christine's b}
b(x)=\begin{cases}(x+1)^2 & x<-\frac{1}{2}\\ -x^2+\frac{1}{2} &
-\frac{1}{2} \le x \le \frac{1}{2}\\ (x-1)^2 & \frac{1}{2} < x.
\end{cases}
\end{equation}
She shows that the Szeg\"o kernel has singularities off of
$\Delta$ in this case.

In this paper, we consider certain non-pseudoconvex domains of the
form \eqref{tube domain} for which $b$ is smooth, real-valued, but
not convex.  One checks that such a domain fails to be
pseudoconvex if and only if $b$ fails to be convex.  More
specifically, we take
\begin{equation}\label{our b's}
b(x)=\frac{1}{4}x^4 + \frac{1}{2}px^2+qx,\quad  p<0, q\in
\mathbb{R}.
\end{equation}
(Note that the condition on $p$ is precisely the one required for
a quartic of this form to be non-convex.) Our goal is to identify
sets in $\C{2} \times \C{2}$ on which the integrals defining the
Szeg\"o kernel and its derivatives are absolutely convergent.

\section{Definitions, Notation, and Statement of Results}

We begin with a more precise discussion of the Szeg\"o projection
operator and its associated integral kernel for domains in $\C{2}$
having the form \eqref{tube domain}. We take $b$ smooth so that
$\Omega \subset \C{2}$ is smoothly-bounded. As above, let
$\mathcal{O}(\Omega)$ denote the space of functions holomorphic on
$\Omega$.  Define
$$\mathcal{H}^2(\Omega):=\left\{\,F\in \mathcal{O}(\Omega): \sup_{\varepsilon > 0} \int_{\partial \Omega} |F(x+iy, t+ib(x)+i\varepsilon)|^2\, dx\,dy\,dt<\infty\,\right\}.$$
$\mathcal{H}^2(\Omega)$ can be identified with the space of $f \in
L^2(\partial \Omega)$ which are solutions (in the sense of
distributions) to
\begin{equation}\label{CR equation}
\left(\pdd{x} + i\pdd{y} -ib'(x) \pdd{t} \right)[f] \equiv 0.
\end{equation}
With this identification, we define the {\it Szeg\"o projection
operator $\mathcal{S}$} to be the orthogonal projection of
$L^2(\partial \Omega)$ onto this (closed) subspace
$\mathcal{H}^2(\Omega)$.

One can then prove the existence of an integral kernel associated
with the operator.  This is discussed, for example, in
\cite{St(BB):72}, where the approach is as follows: Begin with an
orthonormal basis $\{\phi_j\}$ for $\mathcal{H}^2(\Omega)$ and
form the sum
$$S(z,w)=\sum_{j=1}^{\infty} \phi_j(z)\overline{\phi_j(w)}.$$
One shows that this converges uniformly on compact subsets of
$\Omega \times \Omega$, that $\overline{S(z , \cdot)}\in
\mathcal{H}^2(\Omega)$ for each $z \in \Omega$, and that for $g
\in \mathcal{H}^2(\Omega)$,
$$g(z)=\int_{\partial \Omega} S(z,w) g(w)\,d\sigma(w).$$
$S$ is then the {\it Szeg\"o kernel}.  From its construction it is
clear that it will be smooth on $\Omega \times \Omega$.  It may
extend to a smooth function on some larger subset of
$\overline{\Omega} \times \overline{\Omega}$.

For domains of the form \eqref{tube domain}, one can derive an
explicit formula for the Szeg\"o kernel.  Let $z=(z_1, z_2)$ and
$w=(w_1, w_2)$ be elements of $\C{2}$.  Set
\begin{equation}\label{N(eta,tau)}
N(\eta, \tau)=\int_{-\infty}^\infty e^{2\tau[\eta \lambda -
b(\lambda)]}\,d\lambda.
\end{equation}
Then
\begin{equation}\label{Szego kernal inside}
S(z,w)=c\int \!\!\!\int_{\tau > 0} \tau e^{\eta \tau [z_1 +
\bar{w_1}]+i\tau[z_2 - \bar{w_2}]} [N(\eta, \tau)]^{-1} \,d\eta
\,d\tau,
\end{equation}
where $c$ is an absolute constant.
\begin{remark}
See \cite{HNW:09} for detailed discussions of $\mathcal{H}^p$
spaces for unbounded domains, the derivations of such integral
formulas, and the identification of $\mathcal{H}^2(\Omega)$ with
$L^2(\Omega)$ functions satisfying the differential equation
\eqref{CR equation}.
\end{remark}
\begin{remark}
Many authors only consider $S$ as a distribution on $\partial
\Omega \times
\partial \Omega$ since $S$ is smooth on $\Omega \times \Omega$.  In this situation, one can identify the boundary with $\R{3}$ and
consider the integral kernel
\begin{eqnarray}\label{Szego kernel on R3 X R3}
\lefteqn{\mathcal{S}[(x,y,t),(r,s,u)]=}\nonumber\\
& &c\int_0^\infty \!\!\int_{-\infty}^\infty \tau e^{\tau[i(t-u)+i
\eta (y-s)-[b(x)+b(r)-\eta(x+r)]]}\left[N(\eta, \tau)\right]^{-1}
\,d\eta \,d\tau.\label{szego kernel for tube}
\end{eqnarray}
This is done, for example, in the work of Nagel
\cite{Nag(Beijing):86}, Haslinger \cite{Haslinger:95}, and
Carracino \cite{CarracinoPHD}, \cite{Carracino:07}.
\end{remark}

We may now state our results:

Let $b$ be as in \eqref{our b's} and let
\begin{gather}
z=(z_1,z_2)=(x+iy, t+ib(x)+i h)\\
w=(w_1,w_2)=(r+is,u+ib(r)+i k).
\end{gather}
Define
\begin{equation}\label{sigma}
\Sigma=\{\,(z,w): x=r \;\text{and}\; |x| > \sqrt{-p}\,\}\cup
\{\,(z,w): |x|=|r|=\sqrt{-p} \,\}.
\end{equation}
Also, for a continuous function $b$ on $\mathbb{R}$, define the
{\it Legendre transform} of $b$ by
\begin{equation}
b^*(\eta):=\sup_{\lambda \in \mathbb{R}} [\eta \lambda -
b(\lambda)].
\end{equation}
\begin{theorem}\label{theorem: absolute convergence}
The integral defining $S(z,w)$ is absolutely convergent in the
region in which
\begin{equation}
h+k+b(x)+b(r)-2b^{**}\left(\frac{x+r}{2}\right)>0.
\end{equation}
This is an open neighborhood of $(\overline{\Omega}\times
\overline{\Omega})\setminus \Sigma$. More generally, if $i_1$,
$j_1$, $i_2$, and $j_2$ are non-negative integers, then
\begin{equation}\label{derivatives of S}
\partial^{i_1}_{z_1} \partial^{j_1}_{\bar{w_1}} \partial^{i_2}_{z_2}
\partial^{j_2}_{\bar{w_2}}S(z,w)=c' \int \!\!\! \int_{\tau>0} e^{\eta \tau
[z_1+\bar{w_1}]+i\tau[z_2-\bar{w_2}]}\frac{\eta^{i_1+j_1}
\tau^{i_1+j_1+i_2+j_2+1}}{N(\eta,\tau)}\,d\eta\,d\tau
\end{equation}
is absolutely convergent in the same region.
\end{theorem}
\begin{remark}
Compare this with Theorem 3.2 in \cite{HNW:09}.  In that theorem,
the domain is of the form \eqref{tube domain} for $b$ convex, and
the region in which the integrals converge absolutely is defined
by the inequality
$$h+k+b(x)+b(r)-2b\left(\frac{x+r}{2}\right)>0.$$
These two theorems are, in fact, analogous since the Legendre
transform is an involution on the set of {\it convex} functions.
\end{remark}
\begin{theorem}\label{theorem: x=-r singularity}
If $[(x+i0,0+ib(x)),(r+i0,0+ib(r))] \in \Sigma$,
$\mathcal{S}[(x,0,0),(r,0,0)]$ is infinite. Also, if $\delta =
h+k> 0$,
$$\lim_{\delta \to 0^+} S[(x,i(b(x)+h)),(r,i(b(r)+k))]=\infty.$$
\end{theorem}
Since these theorems show that the nature of the singular set for
the Szeg\"{o} kernel can be different in the non-pseudoconvex case
from what has been observed in the pseudoconvex case, we summarize
these differences in the following:
\begin{cor}
If the domain $\Omega$ is not pseudoconvex, there may be points on
the diagonal of the boundary at which the Szeg\"o kernel is {\it
not} singular and points {\it off the diagonal} at which the
kernel {\it is} singular.
\end{cor}

An understanding of the Szeg\"o kernel requires sharp estimates of
the integral $N(\eta,\tau)$.  For fixed $\eta,\tau$, this is an
integral of the form $\int_{-\infty}^\infty
e^{\rho(\lambda)}\,d\lambda$, where $\rho$ satisfies
$\lim_{|\lambda| \to \infty} \rho(\lambda) = -\infty$. The
heuristic principle that guides the analysis of such integrals is
that the main contribution comes from a neighborhood of the
point(s) at which the exponent attains its global maximum. If
$\lambda_0$ is such a point,
$$\int_{-\infty}^\infty e^{\rho(\lambda)} \,d\lambda = e^{\rho(\lambda_0)} \int_{-\infty}^\infty e^{\rho(\lambda)-\rho(\lambda_0)}\,d\lambda= e^{\rho(\lambda_0)} \int_{-\infty}^{\infty} e^{-p(x)} \,dx,$$
where $p(x):=\rho(\lambda_0)-\rho(x+\lambda_0)$ is non-negative
and vanishes to second order at the origin. In Section
\ref{Section:B}, we focus on understanding the ``main
contribution" to the integral $N$, while in Section 4, we focus on
uniform estimates on the integral that remains once we have taken
out this contribution. The theorems are established in Sections 5
and 6.

\section{\label{Section:B} Behavior of $\eta \lambda - b(\lambda)$}

Define
$$B_\eta(\lambda):= \eta \lambda - b(\lambda) $$
and consider $b^*(\eta)= \sup_{\lambda} B_\eta(\lambda)$, the
Legendre transform of $b$. Since $B_\eta$ is a polynomial of
degree 4 in $\lambda$ with negative leading coefficient, it tends
to $-\infty$ as $|\lambda| \to \infty$.  It follows that the
supremum is achieved at some $\lambda$ at which $B'_\eta$
vanishes; i.e., at some $\lambda$ satisfying $\lambda^3 + p
\lambda-(\eta-q)=0$. Note that this is a depressed cubic equation.
Therefore, by considering its discriminant, one finds:
\begin{prop}\label{prop: number of roots of B'}
\begin{enumerate}
\item If $4(-p)^3 \le 27(\eta - q)^2$, there is a single $\lambda$
at which $B'_\eta$ changes sign, hence $B_\eta$ has a single local
extremum, which is necessarily the location of the global maximum.

\item If $ 4 (-p)^3 > 27(\eta -q)^2$, $B'_\eta(\lambda)=0$ has
three distinct solutions.  Two correspond to local maxima of
$B_\eta$.  We label them $\lambda_- (\eta)$ and $\lambda_+(\eta)$,
with $\lambda_-(\eta)<\lambda_+(\eta)$.
\end{enumerate}
\end{prop}

The next propositions contain more specific information about the
location(s) of the global maximum of $B_\eta$.
\begin{prop}\label{prop: aux function g}
Let $g(\lambda)=-(\lambda^3 + p\lambda)$.  Then
\begin{enumerate}
\item $g$ is positive on $(-\infty, -\sqrt{-p})$ and $(0,
\sqrt{-p})$, and g is negative on $(-\sqrt{-p}, 0)$ and
$(\sqrt{-p}, \infty)$.

\item $g$ increases on $\left(-\sqrt{-p/3}, \sqrt{-p/3} \right)$
and decreases on $(-\infty, -\sqrt{-p/3})$ and $(\sqrt{-p/3},
\infty)$.
\end{enumerate}
\end{prop}
\begin{proof}
The proof is simple calculus and is omitted.
\end{proof}
\begin{prop} \label{prop: maxima of B} Let $B_\eta(\lambda)=\eta \lambda - b(\lambda)$, with
$b$ as in \eqref{our b's}.
\begin{enumerate}
\item[(i)] If $\eta - q = 0$, $\lambda_-(\eta)=-\sqrt{-p}$,
$\lambda_+(\eta)=\sqrt{-p}$, and
$B_\eta(\lambda_-(\eta))=B_\eta(\lambda_+(\eta))$.  In other
words, the global maximum of $B_\eta$ is achieved at two distinct
points.

\item[(ii)] If $0 < \eta - q < \left(\frac{4(-p)^3}{27}
\right)^{\frac{1}{2}}$,
$$-\sqrt{-p} < \lambda_-(\eta) < 0 < \sqrt{-p}<\lambda_+(\eta) $$
and $B_\eta(\lambda_+)> B_\eta(\lambda_-)$.

\item[(iii)] If $\left(\frac{4(-p)^3}{27} \right)^{\frac{1}{2}}\le
\eta - q$, $B_\eta$ has a single local (hence global) maximum at
$\lambda_+(\eta)>\sqrt{-p}$, and $\lambda_+(\eta) \sim
\eta^{\frac{1}{3}}$ as $\eta \to \infty$.

\item[(iv)] If $-\left(\frac{4(-p)^3}{27} \right)^{\frac{1}{2}}<
\eta - q < 0$,
$$\lambda_-(\eta)< -\sqrt{-p} < 0 < \lambda_+(\eta)< \sqrt{-p} $$
and $B_\eta (\lambda_-)>B_\eta(\lambda_+)$.

\item[(v)] If $\eta - q < -\left(\frac{4(-p)^3}{27}
\right)^{\frac{1}{2}}<0$, $B_\eta$ has a single local (hence
global) maximum at $\lambda_-(\eta)<-\sqrt{-p}$, and
$\lambda_-(\eta) \sim \eta^{\frac{1}{3}}$ as $\eta \to -\infty$.

\end{enumerate}
\end{prop}
\begin{proof}
(i):  If $\eta=q$, then the local extrema of $B_\eta$ occur at
solutions to $g(\lambda)=0$. The three solutions are $\lambda =
-\sqrt{-p}, \, 0 , \,\sqrt{-p}$, and the local maximum is attained
at $\lambda = \pm \sqrt{-p}$. Since in this case $B_\eta(\lambda)
= -\frac{1}{4}\lambda^4-\frac{1}{2}p \lambda^2$ is even, the
conclusion follows.

(ii): By Proposition \ref{prop: number of roots of B'}, the upper
bound on $\eta$ guarantees that $B_\eta$ in fact has two local
maxima. Since $\eta - q
> 0$, for $\lambda \in [0, \sqrt{-p}]$,
$$B'_\eta(\lambda)=(\eta - q)-(\lambda^3 + p\lambda)=(\eta - q) + g(\lambda)>0.$$
Since $g$ is decreasing for $\lambda > \sqrt{\frac{-p}{3}}$,
$(\eta -q)+g(\lambda)=0$ has precisely one solution in
$(\sqrt{-p}, \infty)$, and it is the location of a local maximum
for $B_\eta$. We have named this point $\lambda_+(\eta)$.  On the
other hand, since $(\eta-q)+g(\lambda)$ is also positive on
$(-\infty, -\sqrt{-p}]$, the second local maximum
$\lambda_-(\eta)$ is in $(-\sqrt{-p}, 0)$.

Now, $\eta - q > 0$ and $\lambda_- <0$ imply $(\eta - q)\lambda_-
< (\eta - q)(-\lambda_-)$.  Since $\lambda_+$ is the location of
the global maximum of the restriction of $B_\eta$ to the positive
real axis, $B_\eta(\lambda_-)=(\eta -
q)\lambda_--\left(\frac{1}{4}\lambda_-^4 + \frac{1}{2}p
\lambda_-^2 \right) < (\eta - q)(-\lambda_-)-\left(
\frac{1}{4}(-\lambda_-)^4 +\frac{1}{2}p(-\lambda_-)^2\right)=
B_\eta(-\lambda_-) < B_\eta(\lambda_+)$. This proves (ii).

(iii): By Proposition \ref{prop: number of roots of B'}, we are in
the situation in which $B'_\eta(\lambda)=0$ has a single solution.
An identical argument to the one used to prove (ii) shows that the
solution, which we call $\lambda_+(\eta)$, satisfies $\sqrt{-p}<
\lambda_+(\eta)$.

We now prove the statement about the asymptotic behavior of
$\lambda_+(\eta)$. Since $\lambda_+^3 > \lambda_+^3 + p \lambda_+
= \eta -q$, $\lambda_+(\eta) \to \infty$ as $\eta \to \infty$.
Also, since $\lambda_+^3 = \eta - q - p \lambda_+ ,$ we have
$$1=\frac{\eta}{\lambda_+^3}+o(1).$$
Thus $\lambda_+^3 \sim \eta$, i.e., $\lambda_+^3 = \eta[1 + o(1)]$
as $\eta \to \infty$. It follows that $\lambda_+(\eta) \sim
\eta^{\frac{1}{3}}$ as $\eta \to \infty$.

The proofs of (iv) and (v) are almost identical to the proofs of
(ii) and (iii) and are omitted.
\end{proof}

Define a function
\begin{equation}\label{location of global max}
\lambda(\eta)=\begin{cases}\lambda_-(\eta) & \eta < q\\
\sqrt{-p} & \eta = q\\ \lambda_+(\eta) & \eta > q.
\end{cases}
\end{equation}
Thus for $\eta \ne q$, $\lambda(\eta)$ is the location of the global
maximum of $B_\eta$.  For $\eta = q$, the global maximum is achieved
at two points, $\pm\sqrt{-p}$.  Which of these we choose for the
value of $\lambda(q)$ is arbitrary.
\begin{prop} The function
$\eta \mapsto \lambda(\eta) $ maps $\mathbb{R}$ onto
$\mathbb{R}\setminus [-\sqrt{-p}, \sqrt{-p})$. Furthermore, it is
\begin{enumerate}
\item[(a)] differentiable on $\mathbb{R}\setminus
\{q\}$,
\item[(b)] continuous from the right at $\eta = q$, and
\item [(c)]increasing and injective on $\mathbb{R}$.
\end{enumerate}
\end{prop}
\begin{proof}
The equation $\eta = q + \lambda^3 + p\lambda$ clearly expresses
$\eta$ as a function of $\lambda$. Furthermore, the restriction of
this function to $(-\infty, -\sqrt{-p}) \cup [\sqrt{-p}, \infty)$
is easily seen to be one-to-one with image $\mathbb{R}$.  Thus its
inverse function is well-defined on $\mathbb{R}$ and maps this set
to $(-\infty, -\sqrt{-p}) \cup [\sqrt{-p}, \infty)$.  Since
$\lambda$ restricted to $\mathbb{R} \setminus \{q\}$ is the
inverse of a function which is smooth with non-vanishing
derivative on its (restricted) domain, $\lambda$ is itself
continuous and differentiable there, with derivative
$\lambda'(\eta)=\displaystyle{\frac{1}{3[\lambda(\eta)]^2+p}} >0.$
The proposition is established.
\end{proof}

\begin{cor}\label{cor: cont of b*}
The function $\eta \to b^{*}(\eta)$ is continuous on $\mathbb{R}$.
\end{cor}
\begin{proof} This is immediate since $b^*$ is a real-valued
function which is convex on all of $\mathbb{R}$.


\end{proof}

\section{Estimates on $\int_{-\infty}^\infty e^{-p(x)}\,dx$} 
\subsection{Definitions and Notation} Let $p$ be a real polynomial
of even degree with positive leading coefficient. We are
interested in estimates on
\begin{equation}\label{full integral}
\int_{-\infty}^\infty e^{-p(x)}\,dx
\end{equation}
which are uniform in the coefficients of $p$. If $p$ is convex
(i.e., if $p''(x)\ge 0$ for all $x$) with $p(0)=p'(0)=0$, we know
that
\begin{equation}\label{desired estimate}
\int_{-\infty}^\infty e^{-p(x)}\,dx \approx |\{\,x:p(x)\le 1 \,\}|,
\end{equation}
where we use the notation $A \approx B$ to mean that there exists
a constant $c$ such that $\displaystyle{cB \le A \le
\frac{1}{c}B}$. When such inequalities hold, we say that $A$ and
$B$ are {\it comparable}. It will be understood whenever this
notation is used that the underlying constant $c$ is independent
of all important parameters. Thus in our case it is always
independent of the coefficients of the polynomial $p$ and depends
only, perhaps, on the degree of $p$.

Our goal is to extend the estimate \eqref{desired estimate} to the
situation in which $p$ is a fourth-degree polynomial with positive
leading coefficient. By translating, shifting, and reflecting about
the $y$-axis if necessary, we can arrange it so that the (not
necessarily unique) global minimum of the polynomial is zero and
occurs at $x=0$, and so that $p$ is convex for all $x \le 0$. Since
$p''$ has degree 2, if $p$ fails to be convex on all of
$\mathbb{R}$, there is a single interval on which $p''$ is negative.

Thus suppose $p''$ has zeros at $x=A$ and $x=A+C$ where $A,C >0$.
Then there exists $B>0$ so that
\begin{equation}\label{p''}
p''(x)=B(x-A)(x-(A+C)).
\end{equation}
\begin{remark}
One checks easily that if $A=0$, $p$ can not have its global
minimum at $0$ unless $C=0$ as well.  In this case, we would have
$p(x)=Bx^4$, which is convex. Furthermore, if $A> 0$ but $C=0$,
$p''$ is never negative, hence $p$ is convex.
\end{remark}
If we anti-differentiate \eqref{p''} twice, using the assumption
that $p(0)=p'(0)=0$, we find
\begin{equation}\label{p in terms of A, B, C}
p(x)=\frac{B}{12}x^2[x^2-2(2A+C)x+6A(A+C)].
\end{equation}
In the analysis that follows, it will be essential to know what
relationship, if any, exists between $A$ and $C$.  Thus write
$C=\alpha A$ for $\alpha > 0$.  Then
\begin{equation}\label{p in terms of A, B, alpha}
p(x)=\frac{B}{12}x^2[x^2-2A(2+\alpha)x+6A^2(1+\alpha)].
\end{equation}

\begin{prop}\label{prop: condition on C for 0 to be global min} Let $p$ be as in \eqref{p in terms of A, B, alpha}, with $A,B,\alpha>0$.  $p$ is
non-negative if and only if
\begin{equation}\label{condition on C for 0 to be global min}
0 < \alpha \le 1+\sqrt{3}.
\end{equation}
\end{prop}
\begin{proof}
$p$ is non-negative if and only if the expression
$x^2-2A(2+\alpha)x+6A^2(1+\alpha)$ is non-negative for all $x$. The
conclusion follows by finding those positive $\alpha$ for which this
quadratic has non-positive discriminant.
\end{proof}

Next, we prove an inequality concerning the value of $p$ at its
inflection points:
\begin{prop}\label{prop: rel between p(A+C) and p(A)} If
$p(x)=\frac{B}{12}[x^4-2A(2+\alpha)x^3+6A^2(1+\alpha)x^2]$ and $0 <
\alpha \le 1+\sqrt{3}$, then there exists $c>0$ independent of $A$
and $B$ so that $p((1+\alpha)A)\ge p(A) \ge cBA^4$.
\end{prop}
\begin{proof}
We have
$p((1+\alpha)A)=\frac{BA^4}{12}(3+8\alpha+6\alpha^2-\alpha^4)$ and
$p(A)=\frac{BA^4}{12}(3+4\alpha)$. The lower bound on $p(A)$
follows immediately since for $0 < \alpha \le 1+\sqrt{3}$,
$3+4\alpha$ is bounded below by a positive constant.

Observe that
$p((1+\alpha)A)-p(A)=\frac{B}{12}A^4(4\alpha+6\alpha^2-\alpha^4)$.
One confirms easily that $\alpha=-2, 0, 1\pm\sqrt{3}$ are roots.
Finally, since at $\alpha=1$,
$p((1+\alpha)A)-p(A)=\frac{B}{12}A^4(4+6-1)>0$, we conclude that
this difference is positive for all $0<\alpha<1+\sqrt{3}$.  This
proves the proposition.
\end{proof}

A convex polynomial clearly has only one local extremum, which is
necessarily the location of the global minimum.  For non-convex
$p$, however, it is possible that $p$ has other extrema. More
specifically, if $p$ is a fourth-degree polynomial, $p'$ is a
polynomial of degree three, hence it has either a single real root
or three real roots (counting multiplicities).  We have the
following:
\begin{prop}\label{condition on C for 3 extrema}
Let $p(x)=\frac{B}{12}[x^4-2A(2+\alpha)x^3+6A^2(1+\alpha)x^2]$,
with $A,B>0$ and $0<\alpha\le 1+\sqrt{3}$.  Then $p'$ has three
real roots if and only if $2 \le \alpha \le 1+\sqrt{3}$.
\end{prop}
\begin{proof}
We find that $p'(x)=
\frac{B}{12}x[4x^2-6A(2+\alpha)x+12A^2(1+\alpha)]$. This has three
real roots if and only if $3A^2(3\alpha^2 -4 \alpha -4)\ge 0$.
This occurs if and only if $\alpha\le-\frac{2}{3}$ or $\alpha \ge
2$. Since we have assumed $0 < \alpha \le 1+\sqrt{3},$ the
conclusion follows.
\end{proof}

To analyze the integral \eqref{full integral}, we begin by writing
it as a sum:
\begin{eqnarray}
\int_{-\infty}^\infty e^{-p(x)}\,dx &=&\int_{-\infty}^0
e^{-p(x)}\,dx+\int_0^A
e^{-p(x)}\,dx\nonumber\\
& &{}+\int_{A}^{(1+\alpha)A}
e^{-p(x)}\,dx+\int_{(1+\alpha)A}^\infty e^{-p(x)}\,dx\nonumber\\
&=&I+II+III+IV \label{four pieces}.
\end{eqnarray}
Observe that $p$ is convex on the intervals of integration for $I$,
$II$, and $IV$. Obtaining sharp estimates on these integrals
requires the results of the next subsection.

\subsection{Some Estimates on Functions on Intervals of Convexity}
We will use the following results repeatedly.  The first gives the
size of the integral of $e^{-p}$ over any interval on which $p$ is
convex. It uses a modification of an argument in Halfpap, Nagel,
and Wainger \cite{HNW:09} proving an analogous estimate if $p$ is
convex on all of $\mathbb{R}$.
\begin{lemma}\label{lemma: mod of HNW} Let $p$ be a polynomial satisfying $\lim_{|x|\to\infty}
p(x)=\infty$.
\begin{enumerate}
\item \label{p' pos and incr} Suppose $p'$ is positive and
increasing on an interval $(x_0,x_f)$, where $x_f$ may equal
$+\infty$.  Suppose further that in the case in which $x_f<\infty$,
$p(x_f)\ge p(x_0)+1$. Then
\begin{equation}
\int_{x_0}^{x_f} e^{-p(x)}\,dx \approx
e^{-p(x_0)}|\{x\in(x_0,x_f):p(x_0)<p(x)<p(x_0)+1\}|.
\end{equation}

\item\label{p' neg and incr} Suppose $p'$ is negative and
increasing on an interval $(x_f,x_0)$, where $x_f$ may equal
$-\infty$.  Suppose further that in the case in which $x_f>-\infty$,
$p(x_f)\ge p(x_0)+1$. Then
\begin{equation}
\int_{x_f}^{x_0} e^{-p(x)}\,dx \approx
e^{-p(x_0)}|\{x\in(x_f,x_0):p(x_0)<p(x)<p(x_0)+1\}|.
\end{equation}

\item \label{p' incr, p(T_f) small}  Suppose $p'$ is (i) positive
and increasing on $I=(x_0,x_f)$ with $x_f<\infty$ or (ii) negative
and increasing on $I=(x_f,x_0)$ with $x_f>-\infty$. Suppose further
that $p(x_0) < p(x_f) < p(x_0)+1$. Then
\begin{equation}
\int_I e^{-p(x)}\,dx \approx e^{-p(x_0)}|x_f - x_0|.
\end{equation}
\end{enumerate}
\end{lemma}
\begin{proof}
We sketch the proof of (1).  The remaining parts follow in a
similar manner.

Suppose $x_f < \infty$, and let $J$ be the largest positive
integer such that $p(x_f)\ge p(x_0)+J$.  Our hypotheses guarantee
that such a $J$ exists.  For each positive integer $j \le J$,
define $x_j$ to be the unique element of $(x_0, x_f)$ for which
$p(x_j) = p(x_0) + j$.  Clearly,
$$e^{-p(x_0)}(x_1 - x_0) \le \int_{x_0}^{x_f} e^{-p(x)}\,dx.$$
For the reverse inequality, observe that
\begin{eqnarray*}
\int_{x_0}^{x_f}  e^{-p(x)}\,dx&=&\sum_{j=0}^{J-1} \int_{x_j}^{x_{j+1}} e^{-p(x)}\,dx + \int_{x_J}^{x_f} e^{-p(x)}\,dx\\
&\le& \sum_{j=0}^{J-1} e^{-p(x_0)- j} (x_{j+1} - x_j)+e^{-p(x_0) -J}(x_f - x_J)\\
&\le& e^{-p(x_0)} \left[  (x_1 - x_0) + \sum_{j=1}^{J-1}
e^{-j}(x_{j+1} - x_1)   + e^{-J}(x_f - x_1) \right].
\end{eqnarray*}
We now estimate $x_{j+1} - x_1$ in terms of $x_1 - x_0$.
\begin{equation*}
j=p(x_{j+1}) - p(x_1) =\int_{x_1}^{x_{j+1}} p'(x)\,dx \ge
p'(x_1)(x_{j+1} - x_1).
\end{equation*}
Since $$p'(x_1)(x_1 - x_0) \ge \int_{x_0}^{x_1}p'(x) \,dx = 1,$$
we have $x_{j+1} - x_1 \le j(x_1 - x_0)$. A similar estimate holds
for $x_f - x_1$.  It follows that
$$\int_{x_0}^{x_f} e^{-p(x)} \,dx \lesssim (x_1 - x_0)e^{-p(x_0)}.$$
\end{proof}

\begin{lemma}[Bruna, Nagel, Wainger \cite{BrNagWa:88}]\label{lemma: Bruna, Nagel,
Wainger}  Let $p$ be a polynomial of degree $m$ satisfying
$p(0)=p'(0)=0$; i.e., $p(x)=\sum_{k=2}^m a_k x^k$.  If $p$ is
convex on an interval $[0,A]$, then there exists a constant $C_m$,
depending on $m$ but independent of $A$, such that
\begin{equation}\label{conclusion of BNW}
p(x)\ge C_m \sum_{k=2}^m |a_k| x^k \quad \text{for all $x\in
[0,A]$}.
\end{equation}
\end{lemma}
This lemma is useful to us because it allows us to prove the
following:
\begin{prop}\label{prop: size of mu s.t. p(mu)=1}
Let $p$ be as in Lemma \ref{lemma: Bruna, Nagel, Wainger}. Suppose
that $p(A)>1$.  Then $p(x)=1$ has a unique solution $\mu$ in
$[0,A]$ and
\begin{equation}\label{size of mu}
\mu \approx \left[\sum_{k=2}^m |a_k|^{1/k} \right]^{-1}.
\end{equation}
\end{prop}
\begin{proof}
This is a standard argument, included here for completeness.

It follows from Lemma \ref{lemma: Bruna, Nagel, Wainger} that
there exists $C_m$ such that for all $x\in [0,A]$ ,
$$C_m
\sum_{k=2}^m |a_k|x^k \le \sum_{k=2}^m a_k x^k \le \sum_{k=2}^m
|a_k| x^k .$$ Define $\tilde{p}(x)=\sum_{k=2}^m |a_k|x^k$.  Then
if $y_1$ is the positive solution to $\tilde{p}(x)=1 $ and $y_2$
is the positive solution to $C_m \tilde{p}(x)=1$, then $y_1 \le
\mu \le y_2$. It therefore suffices to show that $y_1$ and $y_2$
are comparable to the expression on the right of \eqref{size of
mu}. We show this for $y_2$.

By definition, $y_2$ satisfies $\displaystyle{\sum_{k=2}^m
C_m|a_k| y_2^k =1}$. Thus for every $k$, $2\le k \le m$, $C_m
|a_k| y_2^k \le 1$, and hence $y_2 \le [C_m^{1/k}
|a_k|^{1/k}]^{-1}$. Since this is true for any $k$, it is true for
the $k_0$ such that
$C_m^{1/{k_0}}|a_{k_0}|^{1/{k_0}}=\displaystyle{\max_{\{2\le k\le
m\}} \,C_m^{1/k}|a_k|^{1/k}}$.  On the other hand,
$$C_m^{1/{k_0}}|a_{k_0}|^{1/{k_0}}\ge \frac{1}{m-1}\sum_{k=2}^m C_m^{1/k}|a_k|^{1/k}.$$
It follows that
\begin{equation}
y_2 \le \left[\frac{1}{m-1}\sum_{k=2}^{m}C_m^{1/k}|a_k|^{1/k}
\right]^{-1}\le\frac{m-1}{C_m^{1/2}}\left[\sum_{k=2}^m |a_k|^{1/k}
\right]^{-1}.
\end{equation}
This gives the desired upper bound on $y_2$.

To obtain a lower bound, let $k_1$ be such that $\displaystyle{C_m
|a_{k_1}|y_2^{k_1}=\max_{\{2\le k \le m\}} \,C_m|a_k|y_2^k }$.
Then
$$(m-1)C_m|a_{k_1}|y_2^{k_1}\ge C_m \tilde{p}(y_2)=1 ,$$
and so
\begin{eqnarray*}
y_2&\ge&\left[(m-1)^{1/{k_1}}C_m^{1/{k_1}}|a_{k_1}|^{1/{k_1}}\right]^{-1}\\
&\ge&\left(\frac{1}{m-1}\right)^{1/2}\left(\frac{1}{C_m}
\right)^{1/m}\left[ |a_{k_1}|^{1/{k_1}}\right]^{-1}\\
&\ge&\left(\frac{1}{m-1}\right)^{1/2}\left(\frac{1}{C_m}
\right)^{1/m}\left[ \sum_{k=2}^m|a_k|^{1/k}\right]^{-1}.
\end{eqnarray*}
We have now proved the desired estimates on $y_2$.  The estimates
on $y_1$ follow by setting $C_m=1$.
\end{proof}
\subsection{Estimates of the integral \eqref{full integral}}
In this section we prove
\begin{lemma} \label{lemma: unif est on int e^p} If $\beta, \delta> 0$ and $p(x)=\beta x^4 + \gamma x^3 + \delta x^2$
attains its global minimum at the origin), then
\begin{equation}\label{estimate on full integral}
\int_{-\infty}^\infty e^{-[\beta x^4 + \gamma x^3 + \delta
x^2]}\,dx \approx
[\beta^{\frac{1}{4}}+|\gamma|^{\frac{1}{3}}+\delta^{\frac{1}{2}}]^{-1}.
\end{equation}
\end{lemma}
Since the result is already known for convex $p$, it suffice to
establish it for non-convex $p$, taking $\beta =
\displaystyle{\frac{B}{12}}$, $\gamma = -\displaystyle{\frac{B
A(2+\alpha)}{6}}$, and $\delta = \displaystyle{\frac{ B
A^2(1+\alpha)}{2}}$. As in \eqref{four pieces}, we consider this as
a sum of four integrals.
\subsubsection{The integral $I$.}
To estimate $I$, note that
$$q(x)=p(-x)=\frac{B}{12}[x^4+2A(2+\alpha)x^3+6A^2(1+\alpha)x^2] $$
is convex on $(0,\infty)$ with $q(0)=q'(0)=0$. Thus by Lemma
\ref{lemma: mod of HNW} and Proposition \ref{prop: size of mu s.t.
p(mu)=1}, $I$ satisfies the estimate \eqref{estimate on full
integral}, i.e.,
\begin{eqnarray}
I&\approx& \left[\left(\frac{B}{12}\right)^{\frac{1}{4}}+ \left(
\frac{B A}{6}(2+\alpha)\right)^{\frac{1}{3}}+\left(
\frac{B A^2}{2}(1+\alpha)\right)^{\frac{1}{2}}\right]^{-1}\nonumber\\
&\approx& \left[B^{\frac{1}{4}}+ B^{\frac{1}{3}} A^{\frac{1}{3}}+
B^{\frac{1}{2}}A\right]^{-1}\label{est on I}.
\end{eqnarray}
In \eqref{est on I}, we have also used Proposition \ref{prop:
condition on C for 0 to be global min} to conclude that $2+\alpha$
and $1+\alpha$ are both comparable to 1.

Since clearly $I \le I+II+III+IV$, the lemma will follow if we can
show that $II, III, IV \lesssim I$.
\subsubsection{The integral
$II$.} We have two cases, depending on whether $p(A) \ge 1$ or
$p(A)< 1$.

First, if $p(A) \ge 1$, then Lemma \ref{lemma: mod of HNW} and
Proposition \ref{prop: size of mu s.t. p(mu)=1} imply, as they did
in the case of integral $I$, that
\begin{equation}
II \approx \left[B^{\frac{1}{4}}+ B^{\frac{1}{3}} A^{\frac{1}{3}}+
B^{\frac{1}{2}}A\right]^{-1}\approx I\label{est on II; p(A)>1},
\end{equation}
as desired.

Suppose, then, that $p(A) < 1$.  Then by Lemma \ref{lemma: mod of
HNW},
\begin{equation}\label{est on II; p(A) small}
II \approx A.
\end{equation}
By Proposition \ref{prop: rel between p(A+C) and p(A)}, $c BA^4
\le p(A)$, and so if $p(A)< 1$, $B A^4 \lesssim 1$. Thus
\begin{eqnarray*}
A[B^{\frac{1}{4}}+B^{\frac{1}{3}}A^{\frac{1}{3}}+B^{\frac{1}{2}}A]&=&B^{\frac{1}{4}}A+B^{\frac{1}{3}}A^{\frac{4}{3}}+B^{\frac{1}{2}}A^2\\
&=&(BA^4)^{\frac{1}{4}}+(BA^4)^{\frac{1}{3}}+(BA^4)^{\frac{1}{2}}\\
&\lesssim&1.
\end{eqnarray*}
It follows from \eqref{est on II; p(A) small} that $II \lesssim
I$.

\subsubsection{The integral $III$.} This is the integral over the
interval on which $p''$ is negative. This forces the minimum of
$p$ on this interval to be either $p((1+\alpha)A)$ or $p(A)$. By
Proposition \ref{prop: rel between p(A+C) and p(A)}, both are
bounded below by $cB A^4$ for some uniform positive constant $c$.
Therefore
\begin{equation}\label{est on III}
III \le \alpha A e^{-c B A^4} \lesssim A e^{-c B A^4}.
\end{equation}
This contribution is always less than that from the integral $I$.
Indeed,
$$Ae^{-cB
A^4}[B^{\frac{1}{4}}+B^{\frac{1}{3}}A^{\frac{1}{3}}+B^{\frac{1}{2}}A]
=[(BA^4)^{\frac{1}{4}}+(BA^4)^{\frac{1}{3}}+(BA^4)^{\frac{1}{2}}]e^{-cBA^4}$$
is uniformly bounded since the function
$f(x)=(x^{\frac{1}{4}}+x^{\frac{1}{3}}+x^{\frac{1}{2}})e^{-cx}$ is
bounded on the positive real axis.

\subsubsection{The integral $IV$.} As with integrals $I$ and $II$,
we are integrating over an interval on which $p$ is convex. In
order to use Lemma \ref{lemma: mod of HNW}, we need to know the
minimum value of $p$ on this interval.  We distinguish two cases.

First, suppose the minimum occurs at $x=(1+\alpha)A$. Note that this
implies that $p'((1+\alpha)A)\ge 0$. We must find
$$|\{\,x>(1+\alpha)A: p[(1+\alpha)A]\le p(x) \le p[(1+\alpha)A]+1\,\}|. $$
If $y$ is the unique solution to $p(y)=p[(1+\alpha)A]+1$ in this
interval, then the desired measure is $\nu = y - (1+\alpha)A$.
Recall that $p$ has an inflection point at $x=(1+\alpha)A$ and
expand $p$ about $(1+\alpha)A$ to obtain
$$p[(1+\alpha)A]+p'((1+\alpha)A)(x-(1+\alpha)A)+\frac{\alpha BA}{6}(x-(1+\alpha)A)^3+\frac{B}{12}(x-(1+\alpha)A)^4. $$
Thus $\nu$ is the solution to
\begin{equation}\label{nu if p(A+C) min on (A+C,infty)}
p'((1+\alpha)A)\nu +\frac{\alpha BA}{6}\nu^3+\frac{B}{12}\nu^4=1.
\end{equation}
It follows that the solution to \eqref{nu if p(A+C) min on
(A+C,infty)} is less than the $\tilde{\nu}$ satisfying
\begin{equation}
\frac{\alpha B A}{6}\tilde{\nu}^3+\frac{B}{12}\tilde{\nu}^4=1.
\end{equation}
Since $\alpha, A, B, \tilde{\nu}>0$, $\nu \le \tilde{\nu} \approx
\left[B^{\frac{1}{4}}+B^{\frac{1}{3}}(\alpha A)^{\frac{1}{3}}
\right]^{-1}$. Thus
\begin{equation}\label{est on IV if min at A+C}
IV\lesssim e^{-c BA^4}\left[B^{\frac{1}{4}}+B^{\frac{1}{3}}(\alpha
A)^{\frac{1}{3}}\right]^{-1} .
\end{equation}
We claim that $IV \lesssim I$.  Indeed, since $\alpha>0$,
\begin{eqnarray*}
e^{-cBA^4}\frac{B^{\frac{1}{4}}+B^{\frac{1}{3}}A^{\frac{1}{3}}+B^{\frac{1}{2}}A}{B^{\frac{1}{4}}+B^{\frac{1}{3}}(\alpha A)^{\frac{1}{3}}}&\le&e^{-cBA^4}\frac{B^{\frac{1}{4}}+B^{\frac{1}{3}}A^{\frac{1}{3}}+B^{\frac{1}{2}}A}{B^{\frac{1}{4}}}\\
&=&e^{-cBA^4}[1+(B A^4)^{\frac{1}{12}}+(B A^4)^{\frac{1}{4}}].
\end{eqnarray*}
Since $f(x)=[1+x^{\frac{1}{12}}+x^{\frac{1}{4}}]e^{-cx}$ is a
bounded function on the positive real axis, the conclusion
follows.

Suppose, next, that the minimum of $p$ on $[(1+\alpha)A,\infty)$
occurs at some point $x_0$ interior to the interval at which $p'$
vanishes. In this case, $p'$ has three distinct real roots, and so
by Proposition \ref{condition on C for 3 extrema}, $2< \alpha \le
1+\sqrt{3}$. Precisely the same argument we used above to show
that, regardless of the size of $p(A)$,
$$\int_{-\infty}^A e^{-p(x)}\,dx\approx \int_{-\infty}^0 e^{-p(x)}\,dx \approx e^{-p(0)}|\{\,x<0:0<p(x)<1\,\}| $$
shows that, regardless of the size of $p[(1+\alpha)A]$,
\begin{equation}\label{IV less that main contrib. . . }
IV=\approx \int_{x_0}^\infty e^{-p(x)}\,dx\approx
e^{-p(x_0)}|\{\,x>x_0:p(x_0)< p(x) <p(x_0)+1\,\}|.
\end{equation}
Thus we must estimate $p(x_0)$ and the positive number $y$
satisfying $p(x_0)+1=p(x_0+y)$. Expanding $p$ in powers of
$y=x-x_0$ yields
\begin{eqnarray}
p(x)&=&p(x_0)+p'(x_0)y+\frac{1}{2}p''(x_0)y^2+\frac{1}{6}p'''(x_0)y^3+\frac{1}{24}p^{(4)}(x_0)y^4\nonumber\\
&=&p(x_0)+\frac{B}{2}[x_0^2-A(2+\alpha)x_0+A^2(1+\alpha)]y^2\nonumber\\
& &{}+\frac{B}{6}[2x_0-A(2+\alpha)]y^3+\frac{B}{12}y^4
\label{expansion about x_0}
\end{eqnarray}
Recall from the proof of Proposition \ref{condition on C for 3
extrema} that
$p'(x)=\frac{B}{6}x[2x^2-3A(2+\alpha)x+6A^2(1+\alpha)]$. Set
\begin{equation}\label{def of epsilon}
\varepsilon=9\alpha^2 - 12\alpha - 12.
\end{equation}
This is positive since $\alpha>2$.  Then
\begin{equation}\label{def of x_0}
x_0=\frac{A}{4}[3(2+\alpha)+\sqrt{\varepsilon}]
\end{equation}
and
\begin{equation*}
x_0^2=\frac{A^2}{8}[9\alpha^2+12\alpha+12+3(2+\alpha)\sqrt{\varepsilon}].\label{x_0^2}
\end{equation*}
Substituting \eqref{def of x_0} and \eqref{def of epsilon} into
\eqref{expansion about x_0} yields
\begin{eqnarray*}
\lefteqn{p(x_0+y)}\\
&=&p(x_0)+\frac{B
A^2}{48}\left[\varepsilon+3(2+\alpha)\sqrt{\varepsilon}
\right]y^2+\frac{B A
}{12}\left[2+\alpha+\sqrt{\varepsilon}\right]y^3+\frac{B}{12}y^4.
\end{eqnarray*}
Thus
$$1=\frac{B A^2}{48}\left[\varepsilon+3(2 + \alpha)\sqrt{\varepsilon}
\right]y^2+\frac{B
A}{12}\left[2+\alpha+\sqrt{\varepsilon}\right]y^3+\frac{B}{12}y^4,
$$
and so
\begin{equation}\label{size of sol to p(x)=p(x_0)+1}
y\approx \left[B^{\frac{1}{4}}+B^{\frac{1}{3}} A^{\frac{1}{3}}(2+
\alpha+\sqrt{\varepsilon})^{\frac{1}{3}}+B^{\frac{1}{2}}A
\varepsilon^{\frac{1}{4}}(\sqrt{\varepsilon}+3(2+\alpha))^{\frac{1}{2}}
\right]^{-1}.
\end{equation}
Since $2< \alpha \le 1+\sqrt{3}$, for such $\alpha$, $0 <
\varepsilon=3(3\alpha^2-4\alpha-4) \lesssim 1$.  Hence
\begin{equation}\label{better size est on y}
y\approx
\left[B^{\frac{1}{4}}+B^{\frac{1}{3}}A^{\frac{1}{3}}+B^{\frac{1}{2}}A\varepsilon^{\frac{1}{2}}
\right]^{-1}.
\end{equation}

Recall that we wish to show that $IV \lesssim I$, or,
equivalently, that
\begin{equation}\label{IV less than I}
e^{-p(x_0)}
\left[B^{\frac{1}{4}}+B^{\frac{1}{3}}A^{\frac{1}{3}}+B^{\frac{1}{2}}A\varepsilon^{\frac{1}{4}}
\right]^{-1} \lesssim
\left[B^{\frac{1}{4}}+B^{\frac{1}{3}}A^{\frac{1}{3}}+B^{\frac{1}{2}}A
\right]^{-1}.
\end{equation}
Since $e^{-p(x_0)}\le 1$ and $\varepsilon \lesssim 1$, this
follows immediately in the case in which $\varepsilon$ is also
bounded below by an absolute constant $\beta$.

To prove \eqref{IV less than I} for all $\varepsilon$, therefore,
it suffices to find an absolute constant $\beta$ such that
\eqref{IV less than I} holds for all $0 < \varepsilon \le \beta$.
Since such an estimate is likely to rely upon the relative
smallness of $e^{-p(x_0)}$ compared to $B A^4$, we need more
information about the size of $p(x_0)$. A calculation using
\eqref{p in terms of A, B, alpha} shows
\begin{eqnarray*}
\lefteqn{p(x_0)}\\
&=&\frac{BA^4}{768}(9\alpha^2+12\alpha+12+3(2+\alpha)\sqrt{\varepsilon})(-3\alpha^2+12\alpha+12-(2+\alpha)\sqrt{\varepsilon})\\
&\approx&BA^4(-3\alpha^2+12\alpha+12-(2+\alpha)\sqrt{\varepsilon}).
\end{eqnarray*}
We claim that there exist positive constants $\beta$ and $d$ such
that for all $\alpha \in (2,1+\sqrt{3}]$, if
$\varepsilon\le\beta$,
$-3\alpha^2+12\alpha+12-(2+\alpha)\sqrt{\varepsilon}\ge d$, from
which it will follow that $p(x_0)\ge dBA^4$.

Indeed, it is easy to see that
$$-3\alpha^2+12\alpha+12-(2+\alpha)\sqrt{\varepsilon}\ge
6(1-\sqrt{\varepsilon}).$$ This is bounded below by 3 if
$\varepsilon \le \frac{1}{4}$. The claim follows.

To prove \eqref{IV less than I} when $\varepsilon \le
\frac{1}{4}$, we must show that
\begin{equation}\label{must be bdd for small epsilon}
e^{-dBA^4}\frac{B^{\frac{1}{4}}+B^{\frac{1}{3}}A^{\frac{1}{3}}+B^{\frac{1}{2}}A}{B^{\frac{1}{4}}+B^{\frac{1}{3}}A^{\frac{1}{3}}+B^{\frac{1}{2}}A\varepsilon^{\frac{1}{4}}}=e^{-dBA^4}\frac{1+(BA^4)^{\frac{1}{12}}+(BA^4)^{\frac{1}{4}}}{1+(BA^4)^{\frac{1}{12}}+(BA^4)^{\frac{1}{4}}(\varepsilon)^{\frac{1}{4}}}
\end{equation}
is bounded. This is indeed the case since
$$0\le f(x)= e^{-dx}\frac{1+x^{\frac{1}{12}}+x^{\frac{1}{4}}}{1+x^{\frac{1}{12}}+x^{\frac{1}{4}} \varepsilon^{\frac{1}{4}}}\le e^{-dx}(1+x^{\frac{1}{12}}+x^{\frac{1}{4}})$$
and the latter is bounded above on the positive real axis.




\subsubsection{Another interpretation.}
Recall that \eqref{desired estimate} holds if $p$ is convex and
$p(0)=p'(0)=0$. We claim that our estimates show that the same is
true in the case of any fourth-degree polynomial with positive
leading coefficient and global minimum at the origin. Indeed, set
$$\mu = |\{\,x:p(x)\le 1\,\}|\quad\quad \mu^- = |\{\,x<0:p(x)\le 1\,\}|.$$
Clearly
$$e^{-1}\mu \le \int_{\{x:p(x)\le 1\}} e^{-p(x)}\,dx \le
\int_{-\infty}^\infty e^{-p(x)}\,dx.$$ On the other hand, the
estimates of the previous section imply the existence of a
constant $C>0$ such that
$$\int_{-\infty}^\infty e^{-p(x)}\,dx \le C \mu^-. $$
Since $\mu^- \le \mu$, it follows that $\int_{-\infty}^{\infty}
e^{-p(x)}\,dx \approx \mu$, as claimed.

\subsection{Remarks on Polynomials of Higher Degree.} The results of this paper can not easily be extended to tube domains \eqref{tube domain}
defined by higher-degree non-convex polynomials $b$ because it is
not clear what uniform estimate should replace Lemma \ref{lemma:
unif est on int e^p}.

Consider, for a moment, the analogue of Lemma \ref{lemma: unif est
on int e^p} for convex polynomials:
\begin{lemma} \label{lemma: for higher convex p} Let $n$ be a positive integer and define
$p(x)=\displaystyle{\sum_{j=2}^{2n} \beta_j x^j}$. Suppose $p$ is
convex on $\mathbb{R}$.  Then
\begin{equation}\label{4.30}
I:=\int_{-\infty}^\infty e^{-p(x)}\,dx \approx
\left[\sum_{j=2}^{2n}|\beta_j|^{\frac{1}{j}} \right]^{-1}.
\end{equation}
\end{lemma}
This lemma is not new; it follows easily from the results of
Bruna, Nagel, and Wainger discussed above.  We saw in Lemma
\ref{lemma: unif est on int e^p} that this same result holds if
$n=2$ even if we replace the hypothesis that $p$ is convex with
the weaker hypotheses that $p$ attains its global minimum at $0$
and $\beta_{2n}>0$.  We claim that such a result does {\it not}
hold if $n=3$.

Indeed, consider $p(x)=x^2(x-a)^4
=x^6-4ax^5+6a^2x^4-4a^3x^3+a^4x^2$, with $a>1$. Clearly $p$ is
non-negative, attains its global minimum at the origin, and is
convex for $x \le 0$. If \eqref{4.30} were true, we would have
both
$$\frac{1}{a^2} \approx [1+a^{\frac{1}{5}}+a^{\frac{1}{2}}+a+a^2]^{-1}
\approx I$$ and
$$ I \ge \int_a^{\infty} e^{-x^2(x-a)^4}\,dx =\int_0^{\infty}
e^{-(y+a)^2y^4}\,dy \approx
[1+a^{\frac{1}{5}}+a^{\frac{1}{2}}]^{-1} \approx
\frac{1}{a^\frac{1}{2}}.$$ (We have used in the above the
observation that $q(y)=(y+a)^2y^4$ is convex on the positive real
axis with global minimum at the origin.)  Since there is no
positive $C$ independent of $a>1$ such that
$\displaystyle{\frac{1}{a^2} \ge \frac{C}{a^{\frac{1}{2}}}}$, our
claim is established.

It is not hard to see what is going on; in the case of a
non-convex fourth-degree polynomial, if there are two competing
global minima, they are both points at which the polynomial
vanishes to order two.  A higher-degree polynomial can have
different orders of vanishing at different competing global
minima.  Thus order of vanishing must be taken into account in the
higher-degree case.

\section{Proof of Theorem \ref{theorem: absolute convergence}}
We now return to the analysis of the integral $N$ in
\eqref{N(eta,tau)}.
\begin{eqnarray*}
&=&e^{2\tau b^*(\eta)} \int_{-\infty}^\infty e^{2\tau[\eta \lambda
-b(\lambda) - B_\eta(\lambda(\eta))]} \,d\lambda\\
&=&e^{2\tau b^*(\eta)}\int_{-\infty}^\infty e^{2\tau [\eta \lambda
-b(\lambda)-\eta\lambda(\eta)+b(\lambda(\eta))]}\,d\lambda\\
&=&e^{2\tau b^*(\eta)}\int_{-\infty}^\infty e^{2\tau
[-b''(\lambda(\eta))\frac{(\lambda-\lambda(\eta))^2}{2}
-b'''(\lambda(\eta))\frac{(\lambda-\lambda(\eta))^3}{6}-\frac{(\lambda
- \lambda(\eta))^4}{4}]}\,d\lambda\\
&=&e^{2\tau b^*(\eta)}\int_{-\infty}^\infty e^{- [2\tau
b''(\lambda(\eta))\frac{y^2}{2}
+2\tau b'''(\lambda(\eta))\frac{y^3}{6}+2\tau\frac{y^4}{4}]}\,dy\\
&\approx&e^{2 \tau b^*(\eta)} \left[\left(\frac{\tau}{2}
\right)^{\frac{1}{4}}+\left|\frac{\tau b'''(\lambda(\eta))}{3}
\right|^{\frac{1}{3}}+\left(\tau b''(\lambda(\eta))
\right)^{\frac{1}{2}} \right]^{-1}\\
&\approx&e^{2 \tau b^*(\eta)}
\left[\tau^{\frac{1}{4}}+\tau^{\frac{1}{3}}
|\lambda(\eta)|^{\frac{1}{3}}+\tau^{\frac{1}{2}} (3\lambda(\eta)^2
+ p)^{\frac{1}{2}} \right]^{-1},
\end{eqnarray*}
where we have used the result of the previous section in the
second-last line.

We now prove Theorem \ref{theorem: absolute convergence}.  If we
show that each integral
\begin{equation}\label{integrals defining derivatives of S}
\int \!\!\! \int_{\tau>0} e^{\eta \tau
[z_1+\bar{w_1}]+i\tau[z_2-\bar{w_2}]}\frac{\eta^{i_1+j_1}
\tau^{i_1+j_1+i_2+j_2+1}}{N(\eta,\tau)}\,d\eta\,d\tau
\end{equation}
is absolutely convergent when
$$h+k+b(x)+b(r)-2b^{**}\left(\frac{x+r}{2} \right)>0 ,$$
it will follow that the integral in fact is equal to
$\partial^{i_1}_{z_1} \partial^{j_1}_{\bar{w_1}}
\partial^{i_2}_{z_2}
\partial^{j_2}_{\bar{w_2}}S(z,w)$.

Set $\delta =h+k$, $z=(z_1,z_2)=(x+iy,t+ib(x)+ih, r+is,
u+ib(r)+ik)$, $i_1+j_1=n$, $i_1+j_1+i_2+j_2=m$ (so that $m \ge
n$). The integral becomes
\begin{equation}
S^{n,m,\delta}:=\int \!\!\! \int_{\tau>0} e^{\eta \tau
[x+r+i(y-s)]+i\tau[t-u+i(b(x)+b(r)+\delta)]}\frac{\eta^{n}
\tau^{m+1}}{N(\eta,\tau)}\,d\eta\,d\tau,
\end{equation}
which converges absolutely if and only if
\begin{equation}\label{int for abs conv}
\widetilde{S}^{n,m,\delta} :=\int_{-\infty}^\infty
\!\int_{0}^\infty e^{-\tau[\delta +b(x)+b(r)-\eta(x+r)]}
\frac{|\eta|^n \tau^{m+1}}{N(\eta,\tau)} \,d\tau\,d\eta < \infty.
\end{equation}
We see that
\begin{eqnarray*}
\widetilde{S}^{n,m,\delta} &\approx& \int_{-\infty}^\infty \!
\int_{0}^{\infty} e^{-\tau[\delta
+b(x)+b(r)-\eta(x+r)+2b^*(\eta)]}\\
& &{}\times
\left[\tau^{\frac{1}{4}}+\tau^{\frac{1}{3}}|\lambda(\eta)|^{\frac{1}{3}}+\tau^{\frac{1}{2}}(3\lambda(\eta)^2
+p)^{\frac{1}{2}}\right] |\eta|^n \tau^{m+1}\,d\tau\,d\eta.\\
&:=&\mathcal{I}^{n,m,\delta}_1 + \mathcal{I}^{n,m,\delta}_2 +
\mathcal{I}^{n,m,\delta}_3.
\end{eqnarray*}
(Since the superscripts are cumbersome, we will often omit them.)
Furthermore, let $\mathcal{I}_i(\eta)$ denote the integrand of the
$\eta$ integral defining $\mathcal{I}_i$, so that
$$\mathcal{I}_i = \int_{-\infty}^\infty \mathcal{I}_i(\eta)\,d\eta.$$

Set
\begin{equation}\label{A(x,r,eta)}
A(x,r,\eta) = b(x)+b(r)-\eta(x+r)+ 2b^*(\eta).
\end{equation}
Since
\begin{equation}\label{5.x}
A(x,r,\eta)=\sup_{\lambda}[\eta \lambda - b(\lambda)]-[\eta x -
b(x)]+ \sup_{\lambda}[\eta \lambda -b(\lambda)]-[\eta r - b(r)],
\end{equation}
$A$ is non-negative.

Each $\mathcal{I}_i(\eta)$ involves an integral in $\tau$ of the
form
\begin{equation}\label{int first in tau}
\int_0^\infty e^{-\tau [\delta+A(x,r,\eta)]}\tau^a \,d\tau
\end{equation}
which equals
\begin{equation}\label{result of tau int}
c_a\frac{1}{[\delta+A(x,r,\eta)]^{a+1}}  \quad\text{if $\delta
+A(x,r,\eta)> 0.$}
\end{equation}
It is now clear that there are two potential barriers to the
convergence of the full integrals $\mathcal{I}_i$:
\begin{enumerate}
\item insufficient growth of $A$ in $\eta$ at infinity, and

\item vanishing of $\delta + A(x,r,\eta)$ for some finite $\eta$
for certain choices of $x$, $r$, $\delta$.
\end{enumerate}
The next subsections explore these issues and in so doing
establish the theorem.

\subsection{Behavior of $A(x,r,\eta)$ for large $\eta$.}

\begin{lemma} \label{Lemma: asy. of A} Fix $x,r\in\mathbb{R}$.  Then
$A(x,r,\eta) \sim \frac{3}{2} \eta^{\frac{4}{3}}$ as $|\eta| \to
\infty$.
\end{lemma}
\begin{proof}
Recall from Proposition \ref{prop: maxima of B} that
$\lambda(\eta) \sim \eta^{\frac{1}{3}}$ as $|\eta| \to \infty$,
i.e., $\lambda(\eta) = \eta^{\frac{1}{3}}(1+o(1))$ as $|\eta| \to
\infty$.  Thus as $|\eta| \to \infty$,
\begin{eqnarray*}
\lefteqn{A(x,r,\eta)}\\
&=&b(x)+b(r)-\eta(x+r)+2(\eta \lambda(\eta) - b[\lambda(\eta)])\\
&=&b(x)+b(r)-\eta(x+r)\\
& &{}+2\left[\eta^{\frac{4}{3}}(1+o(1))
-\frac{1}{4}\eta^{\frac{4}{3}}(1+o(1))^4 -
\frac{1}{2}p\eta^{\frac{2}{3}}(1+o(1))^2 -
q\eta^{\frac{1}{3}}(1+o(1))\right]\\
&=&\frac{3}{2}\eta^{\frac{4}{3}}+\eta^{\frac{4}{3}}o(1)+O(|\eta|)\\
&=&\frac{3}{2}\eta^{\frac{4}{3}}(1+o(1)).
\end{eqnarray*}
\end{proof}
\begin{remark}
Our arguments can be extended to obtain a generalized asymptotic
expansions for $\lambda(\eta)$ and $A(x,r,\eta)$. See Olver
\cite{Olver}, Section 1.5 for a detailed discussion of such
techniques.
\end{remark}

This lemma, equation \eqref{result of tau int}, and parts (iii)
and (v) of Proposition \ref{prop: maxima of B} allow us to
conclude the following:

\begin{enumerate}
\item $\mathcal{I}^{n,m,\delta}_1(\eta) \sim c
(\eta^{\frac{4}{3}})^{-(\frac{5}{4} + m + 1)} |\eta|^n=c
|\eta|^{-3-\frac{4}{3}m+n}$. Since $m \ge n \ge 0$,
$-3-\frac{4}{3}m + n \le -3$, and so for any fixed $n$, $m$, and
$\delta$, $\mathcal{I}^{n, m, \delta}_1$ is convergent at
infinity.

\item $\mathcal{I}^{n,m,\delta}_2(\eta) \sim c
(\eta^{\frac{4}{3}})^{-(\frac{4}{3} + m + 1)} \cdot
|\eta^{\frac{1}{3}}|^{\frac{1}{3}}=c |\eta|^{-3-\frac{4}{3}m +
n},$ and so each $\mathcal{I}^{n,m,\delta}_2$ is convergent at
infinity.

\item $\mathcal{I}^{n,m,\delta}_3(\eta) \sim c
(\eta^{\frac{4}{3}})^{-(\frac{3}{2} + m+ 1)}\cdot
|\eta|^{\frac{1}{3}}=c |\eta|^{-3-\frac{4}{3}m+n},$ and so each
$\mathcal{I}^{n,m,\delta}_3$ is convergent at infinity.

\end{enumerate}

\subsection{Vanishing of $\delta + A(x,r,\eta)$}

The estimates of the previous sections show that whether or not
the integrals $\mathcal{I}_i$ converge depends upon whether or not
for some fixed $x$, $r$, and $\delta$ the function $\eta \mapsto
\delta+A(x,r,\eta)$ vanishes for some finite $\eta_0$ and, if so,
the behavior of this function near such a point. In particular, we
have proved
\begin{prop}
If for some $x$, $r$, and $\delta$ fixed $$\inf_{\eta}
[\delta+A(x,r,\eta)] > 0, $$ then each
$\mathcal{I}^{n,m,\delta}_i$ is finite.
\end{prop}
Note, moreover, that
\begin{eqnarray*}
\inf_{\eta} [\delta+A(x,r,\eta)]&=&\delta +
b(x)+b(r)-2\sup_{\eta}\left[\eta\left(\frac{x+r}{2}
\right)-b^*(\eta)\right]\\
&=&\delta + b(x)+b(r)-2b^{**}\left(\frac{x+r}{2} \right).
\end{eqnarray*}
(The convexity of $b^*$ guarantees the finiteness of the supremum
in the first line.)  We have thus proved that the integrals
defining the Szeg\"o kernel and all its derivatives converge
absolutely in the region
\begin{equation}\label{region for abs conv}
\delta+b(x)+b(r)-2b^{**}\left(\frac{x+r}{2}\right)>0.
\end{equation}
We do not yet know which $x$, $r$, and $\delta$ are in this set.
We claim first that if $z=(z_1,z_2)=(x+iy,t+ib(x)+ih)\in \Omega$
and $w=(w_1,w_2)=(r+is,t+ib(r)+ik)\in\Omega$, $(z,w)$ is in the
region in $\C{2}$ defined by \eqref{region for abs conv}.  Indeed,
$z,w\in \Omega$ implies $h,k>0$, and hence $\delta=h+k>0$.  It
follows that $\delta + A(x,r,\eta) \ge \delta$, and hence its
infimum over $\eta$ is bounded below by $\delta$ as well.  Thus
the inequality in \eqref{region for abs conv} is satisfied.

To prove the remainder of Theorem \ref{theorem: absolute
convergence}, we must determine which  $(z,w) \in \partial \Omega
\times
\partial \Omega$ are in the region \eqref{region for abs conv}.
For such $(z,w)$, $\delta = 0$.  We thus need to determine all
(fixed) $x$ and $r$ for which $A(x,r,\eta)$ is bounded away from
zero independent of $\eta$.

By \eqref{5.x}, $A$ is a sum of two non-negative functions
\begin{gather}
A_x(\eta):=\sup_{\lambda}(\eta \lambda - b(\lambda))-(\eta x -
b(x))=b^*(\eta)- (\eta x - b(x))\\
A_r(\eta):=\sup_{\lambda}(\eta \lambda - b(\lambda))-(\eta r -
b(r))=b^*(\eta)-(\eta r - b(r)).
\end{gather}
Thus for fixed $x$ and $r$, $A$ vanishes at some $\eta_0$ if and
only if both $A_x(\eta_0)$ and $A_r(\eta_0)$ vanish. Furthermore,
by Corollary \ref{cor: cont of b*}, $\eta \to A(x,r,\eta)$ is
continuous, and by Lemma \ref{Lemma: asy. of A}, $A(x, r, \eta)
\sim c \eta^{\frac{4}{3}}$ as $|\eta| \to \infty$. Thus if for
some fixed $x$ and $r$, $A(x,r,\cdot)$ never vanishes, it is
bounded below by a positive constant for all $\eta$. We thus
identify $(z,w)$ in the region \eqref{region for abs conv} by
identifying pairs $x$ and $r$ for which $A(x,r,\cdot)$ never
vanishes.

\noindent{\it Case 1: $|x|<\sqrt{-p}$ or $|r|<\sqrt{-p}$.} For
definiteness, suppose $|x|<\sqrt{-p}$. $A_x$ could only vanish if
$x$ were such that, for some value of $\eta$, the supremum of
$B_\eta(\lambda)=\eta \lambda - b(\lambda)$ were achieved at $x$.
But Proposition \ref{prop: maxima of B} shows that the supremum of
$B_\eta$ is always achieved at one or more points {\it outside of}
$(-\sqrt{-p},\sqrt{-p})$. This completes the proof in this case.

\noindent{\it Case 2: $|x|,|r| > \sqrt{-p}$, and $x \ne r$.} Since
the map $\eta \mapsto \lambda(\eta)$  maps $\mathbb{R} \setminus
\{q\}$ onto $\mathbb{R} \setminus [-\sqrt{-p},\sqrt{-p}]$ and is
injective, there exists a unique $\eta_1 \ne q$ and a unique
$\eta_2 \ne q$ such that $\lambda(\eta_1)=x$ and
$\lambda(\eta_2)=r$. Since $x \ne r$, $\eta_1 \ne \eta_2$. It
follows that in this case $A(x,r,\cdot)$ never vanishes.

\noindent{\it Case 3: $|x|=\sqrt{-p}$ but $|r|>\sqrt{-p}$.} (A
symmetric argument covers the case $|r|=\sqrt{-p}$ but
$|x|>\sqrt{-p}$.) Then $A_x(\eta)=0$ only at $\eta = q$ whereas
one easily computes that $A_r(q)=\frac{1}{4}(r^2+p)^2 > 0$.  Thus
$A(x,r,\cdot)$ does not vanish.

This completes the proof of Theorem \ref{theorem: absolute
convergence}.\hfill \qed

\section{Proof of Theorem \ref{theorem: x=-r singularity}}

We begin by observing that
\begin{eqnarray*}
\lefteqn{S[(x,i(b(x)+h)),(r,i(b(r)+k))]}\\
&=&c \int \!\!\! \int_{\tau>0} \tau e^{\eta \tau
(x+r)-\tau[b(x)+b(r)+h+k]}
N(\eta,\tau)^{-1}\,d\eta\,d\tau\\
&=&\widetilde{S}^{0,0,\delta},
\end{eqnarray*}
and
\begin{eqnarray*}
\mathcal{S}[(x,0,0),(r,0,0)]&=&c \int \!\!\! \int_{\tau>0} \tau
e^{\eta \tau (x+r)-\tau[b(x)+b(r)]}
N(\eta,\tau)^{-1}\,d\eta\,d\tau\\
&=&\widetilde{S}^{0,0,0},
\end{eqnarray*}
where $\widetilde{S}^{n,m,\delta}$ is as defined in \eqref{int for
abs conv}.  We will shorten the notation for these integrals to
$\widetilde{S}^{\delta}$.  Thus to prove Theorem \ref{theorem:
x=-r singularity}, we must show that
\begin{enumerate}
\item[(i)]  $\widetilde{S}^0$ is divergent, and

\item[(ii)] $\lim_{\delta \to 0^+} \widetilde{S}^{\delta}=\infty$.
\end{enumerate}
It is immediately clear that (ii) will follow from (i) since the
integrand of $\widetilde{S}^\delta$ is non-negative and converges
pointwise and monotonically to the integrand of $\widetilde{S}^0$
as $\delta \to 0^+$. Furthermore, (i) will follow if the
corresponding statement holds for any of the three integrals
$\mathcal{I}^{0,0,0}_i$ (again, abbreviated $\mathcal{I}^{0}_i$).
We will show that
\begin{enumerate}
\item[(iii)]  $\mathcal{I}_1^0$ is divergent.
\end{enumerate}
As we saw in the previous section, $\mathcal{I}_1^{0}$ converges
if $x$ and $r$ are chosen in such a way that $A(x,r,\cdot)$ never
vanishes. Thus in order to establish (iii), we need detailed
information about the behavior of $A$ near values $\eta_0$ for
which $A(x,r,\eta_0)=0$. Recall that if $A(x,r,\eta) \ne 0$,
\eqref{result of tau int} shows that the integrand of
$\mathcal{I}_1^{0}$ is comparable to
\begin{equation}\label{integrand of I_1}
[A(x,r,\eta)]^{-\frac{9}{4}}.
\end{equation}
We prove (iii) by considering the behavior of $A$ in three
subcases.

\subsection{Case 1: $x=r$ and $|x|>\sqrt{-p}$.}
In this case, there exists a unique $\eta_0 \ne q$ such that
$x=r=\lambda(\eta_0)$.

Suppose $\eta \ne \eta_0$ and recall that $\eta =
[\lambda(\eta)]^3 +p \lambda(\eta) + q$ and $\eta_0 = x^3 + p x+q
$ so that
\begin{equation}\label{eta - eta_0}
\eta_0 - \eta = (x - \lambda(\eta))(x^2 + x \lambda(\eta) +
[\lambda(\eta)]^2 + p) .
\end{equation}
Then (suppressing the dependence of $\lambda$ on $\eta$ )
\begin{eqnarray*}
A(x,x,\eta)&=&2A_x(\eta)\\
&=&2 [\eta \lambda - \frac{1}{4}\lambda^4 -\frac{p}{2}\lambda^2
-q\lambda -\eta x + \frac{1}{4}x^4 + \frac{p}{2}x^2 + q x]\\
&=&2(x-\lambda)\left[\frac{1}{4}(x+\lambda)(x^2+\lambda^2) +
\frac{p}{2}(x+\lambda)-(\eta-q)\right]\\
&=&2(x-\lambda)\left[\frac{1}{4}(x+\lambda)(x^2+\lambda^2) +
\frac{p}{2}(x+\lambda)-\lambda^3 - p\lambda\right]\\
&=&\frac{1}{2}(x-\lambda)^2(x^2+2\lambda x +3\lambda^2+2p).
\end{eqnarray*}
We are concerned with how this function varies with $\eta$. We
have the following proposition and corollary:
\begin{prop} \label{prop: expressions bounded below} For $|x| > \sqrt{-p}$ fixed,
\begin{enumerate}
\item[(a)]
$$x^2+x\lambda(\eta)+[\lambda(\eta)]^2+p \ge \begin{cases}-2p &
|x|\ge 2\sqrt{-p}\\
|x|(|x|-\sqrt{-p}) & \sqrt{-p}<|x|<2\sqrt{-p}.\end{cases}$$

\item[(b)] $$x^2 + 2x\lambda(\eta)+3[\lambda(\eta)]^2+2p \ge
\begin{cases}-4p &|x|\ge 3\sqrt{-p}\\(|x|-\sqrt{-p})^2 & \sqrt{-p}< |x| <3\sqrt{-p}.\end{cases}$$

\end{enumerate} Thus both expressions are bounded below by a
positive constant independent of $\eta$.
\end{prop}
\begin{proof}
Recall that $|\lambda(\eta)| \ge \sqrt{-p}$.  Our task in part (a)
is thus to find the global minimum of $f(\lambda)=x^2
+x\lambda+\lambda^2+p$ on $\{\,\lambda : |\lambda| \ge \sqrt{-p}
\,\}$.  There are two cases to consider depending on whether $f$
attains its minimum at a critical point or at
$\lambda=\pm\sqrt{-p}$.

Observe that $f'(\lambda)=x+2\lambda=0$ when $\lambda =
-\frac{1}{2}x$. If $\frac{1}{2}|x| \ge \sqrt{-p}$, this indeed is
the location of the global minimum, which is then seen to be
$\frac{3}{4}x^2 + p \ge -2p$. If $\frac{1}{2}|x| < \sqrt{-p}$, the
global minimum is one of the two quantities $x^2 \pm x \sqrt{-p}
-p + p$, which is in turn $\ge x^2 - |x|
\sqrt{-p}=|x|(|x|-\sqrt{-p})$. This proves (a). The proof of (b)
is similar and is omitted.
\end{proof}
\begin{cor} \label{cor: A near eta_0} If $|x| > \sqrt{-p}$,
$$A(x,x,\eta) \approx (\eta - \eta_0)^2 (1+|\eta|)^{-\frac{2}{3}}.$$
\end{cor}
\begin{proof}
By Proposition \ref{prop: expressions bounded below}, we may write
$$A(x,x,\eta)=(\eta - \eta_0)^2\frac{x^2+2x\lambda(\eta)+3[\lambda(\eta)]^2+2p}{2(x^2+x\lambda(\eta)+[\lambda(\eta)]^2+p)^2}:=(\eta - \eta_0)^2 g(\eta).$$
The proposition shows that $g$ is finite for all $\eta$ and
bounded away from zero. Thus for $\eta$ on any fixed interval
$[-K,K]$, $g(\eta) \approx 1$.  On the other hand, Proposition
\ref{prop: maxima of B} shows that $\lambda(\eta)\sim
\eta^{\frac{1}{3}}$ as $|\eta| \to \infty$, and so $g(\eta)
\approx |\eta^{\frac{1}{3}}|^{-2}$ for $|\eta| > K$. Thus for all
$\eta$, $g(\eta) \approx (1+|\eta|)^{-\frac{2}{3}}.$
\end{proof}
By \eqref{integrand of I_1},
\begin{equation}\label{behavior of I_1(eta)}
\mathcal{I}_1^0(\eta) \approx [(\eta -
\eta_0)^2(1+|\eta|)^{-\frac{2}{3}}]^{-\frac{9}{4}},
\end{equation}
and thus $\mathcal{I}^0_1$ is divergent, establishing (iii) in
this case.


\subsection{Case 2: $x=r$ and $|x|=\sqrt{-p}$.}
Here, the analysis is slightly more delicate because the
discontinuity of $\lambda$ occurs at $\eta = q$, and one of $\pm x
- \lambda(\eta)$ vanishes at $\eta = q$.  In this situation, the
analogue of \eqref{eta - eta_0} above is the relationship
\begin{equation}\label{eta - q}
\eta - q = \lambda(\eta)[\lambda(\eta) - x][\lambda(\eta)+x].
\end{equation}
Furthermore, since $x^2 = -p$, in this case
$A(x,x,\eta)=\frac{1}{2}(x-\lambda)^2(3\lambda - x)(\lambda+x)$.
\begin{prop} Let $A$ be as above.
\begin{enumerate}
\item If $x=\sqrt{-p}$, then for $\eta > q$,
$$A(x,x,\eta) \approx (\eta - q)^2(1+|\eta|)^{-\frac{2}{3}} .$$

\item If $x=-\sqrt{-p}$, then for $\eta < q$,
$$A(x,x,\eta) \approx (\eta - q)^2(1+|\eta|)^{-\frac{2}{3}} .$$

\end{enumerate}
\end{prop}
\begin{proof}
In both cases, it is enough to observe that for the values of
$\eta$ indicated, both $|x+\lambda(\eta)| \ge 2|x| > 0$ and
$|3\lambda(\eta) - x | \ge 2|x|>0$.  We may thus solve \eqref{eta
- q} for $x-\lambda$ and substitute into the expression for $A$.
The estimate then follows, again using the fact (Proposition
\ref{prop: maxima of B}) that $\lambda(\eta) \sim
\eta^{\frac{1}{3}}$ as $|\eta| \to \infty$.
\end{proof}
Thus in this case, as above, $\mathcal{I}^0_1$ is divergent.

\subsection{Case 3: $|x|=|r|=\sqrt{-p}$ but $x=-r$.} The point here is that although $x \ne r$, there is an $\eta_0$
for which $B_{\eta_0}$ achieves its global maximum at both $x$ and
$r$: when $\eta = q$ and  $x=\pm\sqrt{-p}$ and $r=\mp\sqrt{-p}$
(See Proposition \ref{prop: maxima of B}).  In this case,
$A_x(\eta)$ vanishes if and only if $\eta = q$. For $\eta \ne q$
\begin{equation}\label{eqn: A_x if |x|=sqrt(-p)}
A_{\pm\sqrt{-p}}(\eta)=(\eta-q)(\lambda(\eta) \mp
\sqrt{-p})-\frac{1}{4}([\lambda(\eta)]^2+p)^2.
\end{equation}
\begin{prop} \label{prop: A if |x|=sqrt(-p)} If $|x|=\sqrt{-p}$,
$$A(x,-x,\eta)=(\eta - q)h(\eta),$$
where
$$|h(\eta)| \approx (1+|\eta|)^{\frac{1}{3}}.$$
\end{prop}
\begin{proof}
It follows from \eqref{eta - q} that
$$[\lambda(\eta)]^2+p = \frac{\eta - q}{\lambda(\eta)} $$
and so
\begin{eqnarray*}
A(\pm\sqrt{-p}, \mp\sqrt{-p},\eta)&=&2(\eta -
q)\left[\lambda(\eta)-\frac{[\lambda(\eta)]^2+p}{4\lambda(\eta)}
\right]\\
&=&(\eta - q)\frac{3[\lambda(\eta)]^2 - p}{2\lambda(\eta)}\\
&:=&(\eta - q) h(\eta).
\end{eqnarray*}
Since the numerator of $h$ is bounded below by $-4p > 0$ and the
denominator is bounded in absolute value away from zero, it
follows that if we fix an interval $[-K,K]$, $|h(\eta)| \approx 1$
on the interval.

On the other hand, since $\lambda(\eta) \sim \eta^{\frac{1}{3}}$
as $|\eta|\to \infty$, for sufficiently large $K$,
$$|h(\eta)|\approx |\eta|^{\frac{1}{3}} \quad\text{for $|\eta|>K$}.$$
The proposition follows.
\end{proof}
Since the integrand of $\mathcal{I}^0_1$ is $\approx
[|\eta-q|(1+|\eta|)^\frac{1}{3}]^{-\frac{9}{4}}$,
$\mathcal{I}^0_1$ diverges in this case as well. The proof of
Theorem \ref{theorem: x=-r singularity} is now complete.

\bibliographystyle{amsalpha}
\bibliography{references}

\end{document}